\documentclass[12pt]{amsart}
\usepackage[left=1.5in, right=1.5in, top=1.5in, bottom=1.5in]{geometry}                
\newtheorem{thm}{Theorem}[section]

\newtheorem{lem}[thm]{Lemma}

\theoremstyle{definition}

\newtheorem{rem}[thm]{Remark}
\usepackage{graphicx}
\usepackage{amssymb}
\usepackage{amsmath}
\usepackage{epstopdf}
\usepackage{hyperref}
\title[]{Small mass uniqueness in the anisotropic liquid drop model and the critical mass problem}
\author{Emanuel Indrei}
\address{Department of Mathematics and Statistics\\
Sam Houston State University\\
Huntsville, TX \\
USA.}

\begin{document}
\setcounter{page}{1}
\pagenumbering{arabic}
\maketitle

\begin{abstract}
Small mass uniqueness in the anisotropic atomic liquid drop model for all values of the parameters is obtained and the critical mass conjecture is investigated.
\end{abstract}

\section{Introduction}
\subsection{Atomic Model}
Two main ingredients define the binding energy of a set of finite perimeter $E \subset \mathbb{R}^n$ with reduced boundary $\partial^* E$: the surface energy
$$
\mathcal{F}(E)=\int_{\partial^* E} f(\nu_E) d\mathcal{H}^{n-1}
$$
and the nonlocal Coulomb repulsion energy
$$
\mathcal{D}(E)=\mathcal{D}_{\lambda}(E)=\frac{1}{2} \int \int_{E \times E} \frac{1}{|z-y|^\lambda}dz dy.
$$
The binding energy is the sum:
$$
\mathcal{E}(E)=\mathcal{F}(E)+\mathcal{D}(E).     
$$
In the classical context, $\lambda=1$, $n=3$, $f(x)=|x|$. Gamow developed this model of atomic nuclei in his 1930 paper \cite{row52} and it successfully predicts the spherical shape of nuclei and the non-existence of nuclei with a large atomic number. More specifically, $E\subset \mathbb{R}^3$ is a nucleus with constant density assumed to equal one, and the number of nucleons corresponds approximately to $|E|=m$. In the surface energy, the surface tension keeps the  nucleus  together and the repulsion energy encodes the repulsion among the protons. The main feature of the minimal binding energy with a mass constraint 

$$
E(m)=\inf_{|E|=m} \mathcal{E}(E).
$$

\noindent is the inclination of the surface energy to generate a ball and the tendency of the Coulomb repulsion energy to avoid the ball via the fact that the ball maximizes the Coulomb energy with a mass constraint. 
In this way, the existence is not in general anticipated true for any mass yielding the non-existence of nuclei with a large atomic number. In the $30$s, the theory was developed further by von Weizs\"acker \cite{row57} and Bohr \cite{row576}. Recently, the model was connected to: Ohta-Kawasaki, Ginzburg-Landau, Thomas-Fermi, Wigner crystallization, and the Thomas-Fermi-Dirac-von Weizs\"acker electron model \cite{MR3251907, MR3727064}. Moreover, it is interconnected with the quantum ionization conjecture: the number of electrons bound to an atomic nucleus of charge $Z$ cannot exceed $Z+1$ \cite{MR3251907, qknm, MR1098611, MR2018928}.

Assuming small mass, thanks to the scaling, the main term is the surface energy.
Hence in physics literature the expectation is existence of a critical mass $m_*$ such that for $m \le m_*$, the ball is the unique minimizer and for $m>m_*$, non-existence is true \cite{MR3425373, MR4314139}. The critical mass is conjectured to be
$$ 
m_*=5\frac{2^{1/3}-1}{1-2^{-2/3}} \approx 3.512.
$$  
Supposing that $E(m)$ does not have a minimizer when $m>m_*$ and a lower bound on a modulus, the expected conclusion is true, see Theorem \ref{@'r}. The modulus generates a phase change in the case that the lower bound is true. In addition, an upper bound on the modulus holds which is consistent with the condition on the expected lower bound, see Remark \ref{zw7}. Assuming $m$ and $\epsilon>0$ small, the modulus has the form 
$
q(\epsilon)m^{\frac{n-1}{n}},
$
see Theorem \ref{@'}. Theorem \ref{@'rj} yields that $E(m)$ does not have a minimizer when $m>m_{\lambda,n, f_{\mathbb{S}^{n-1}}}$, where the constant is explicit  for $\lambda \in (0,1]$.

The mathematical model encodes more freedom. Assuming $\lambda <n$ is natural to make the Coulomb repulsion energy well-defined. Also, the anisotropic surface energy is more physical via the minimizer of the surface energy not necessarily confined to be the ball. Theorem \ref{@'} proves the small mass uniqueness (mod translations) in this optimal range when $f$ is a surface tension, i.e. a convex positively 1-homogeneous 
$f:\mathbb{R}^n\rightarrow [0,\infty)$,
with $f(x)>0$ if $|x|>0$. In the argument, the central ingredients are: (i) uniqueness of the Wulff shape in the anisotropic isoperimetric inequality proved by Fonseca and M\"uller \cite{MR1130601}; (ii) a scaling argument; (iii) compactness.

 The explicit bound contained in Theorem \ref{@'rj} is near the conjectured mass in the isotropic context (the unit ball in $\mathbb{R}^n$ is denoted by $B$ and the $(n-1)$ unit ball by $B^{n-1}$) 
\begin{equation} \label{ar9}
\text{Conjectured critical mass \cite{MR4314139}}=|B|\Big(\frac{2^{1/n}-1}{1-2^{(\lambda-n)/n}} \frac{\mathcal{H}^{n-1}(\mathbb{S}^{n-1})}{\mathcal{D}(B)}\Big)^{\frac{n}{n-\lambda+1}}
\end{equation}
$$
\text{Obtained mass}=|B|\Big(\frac{|B|}{|B^{n-1}|}\frac{\mathcal{H}^{n-1}(\mathbb{S}^{n-1})}{\mathcal{D}_{\lambda-1}(B)}\Big)^{\frac{n}{n-\lambda+1}}.
$$
The Gamow theory is related to several other minimization problems \cite{ I21, qk, bessas2023shape, MR4373173}. 

\subsection{Main contributions}
\subsubsection{Isotropic tensions} If $\lambda=1$, $n=3$ Choksi and Peletier are attributed with the initial conjecture in more modern language that when a minimizer exists it is a ball \cite{MR2653253}. Kn\"upfer and Muratov obtained \cite{MR3055587, MR3272365}: \\

\noindent (i) if $n=2$, $\lambda$ is small, balls are the unique minimizers for $m \le m_*$ and no minimizer exists for $m>m_*$;\\
(ii) if $n=2$, $\lambda \in (0,2)$ or $3\le n \le 7$, $\lambda \in (0,n-1)$, and the mass is small, balls are the only minimizers;\\
(iii) if $n \ge 2$, $\lambda \in (0,2)$, there is $m_{a}>0$ such that if $m>m_a$, no minimizer exists.\\

\noindent Supposing $\lambda=1$, $n=3$, Otto and Lu \cite{MR3251907} independently showed non-existence for masses above some mass $m_e>0$ which was not explicitly estimated (observe that $m_a$ was also not explicitly estimated).
Muratov and Zaleski \cite{MR3302176} obtained (i) if $\lambda \le .034$. Bonacini and Cristoferi \cite{MR3226747} proved (i) and (ii) for a general dimension $n \ge 2$. Figalli, Fusco, Maggi, Millot, and Morini \cite{MR3322379} obtained (ii) in the optimal range $\lambda \in (0,n)$ \& $n\ge 2$ and also a uniqueness theorem was obtained by Julin in \cite{MR3218265}. Frank and Nam \cite{MR4314139} proved existence for $m\in (0,m_*]$, $\lambda \in (0,n)$, and $n \ge 2$. Assuming $E(m)$ does not have a minimizer for $m>m_*$, balls were proved to be the unique minimizers when $m<m_*$; furthermore, they were proved to also be minimizers when $m=m_*$, cf. Frank and Lieb \cite{MR3425373} \& \cite{MR3981098}. The main new element of Theorem \ref{@'r} is the modulus which encodes also the $m=m_*$ uniqueness.
Supposing $\lambda \le 2$ \& $\lambda \in (0,n)$, Frank and Nam proved non-existence with a nonexplicit range \cite{MR4314139}; assuming $\lambda=1$, $n=3$, the non-existence was obtained by Frank, Killip, and Nam for $m>8$ \cite{qkn}.
\subsubsection{Anisotropic tensions} The main existence of bounded minimizers assuming $m$ small was proved by Choksi, Neumayer, and Topaloglu \cite{MR4385589}. Moreover, the non-existence was also proved assuming $\lambda \in (0,2)$ for masses above some mass $m_e>0$ which was not explicitly estimated. In the case of a specific crystalline tension, uniqueness was shown when the mass is small. Misiats and Topaloglu \cite{qko} investigated convergence of scaled minimizers to the Wulff shape. Assuming a weighted perimeter, existence of bounded minimizers was proved by Alama, Bronsard, Topaloglu, and Zuniga in \cite{MR4172686}. \\
\vskip 1.3in

\subsection{Results}

\begin{thm} \label{@'}
If $\lambda \in (0,n)$, $n \ge 2$, there exists $m_0=m_0(n,\lambda,f)>0$ and a modulus of continuity $q(0^+)=0$ such that for all $m<m_0$ there exists $\epsilon_0>0$ such that for all $0<\epsilon<\epsilon_0$ $\&$ for all minimizers $E_m \subset B_R$, $E \subset B_R$, $|E|=|E_m|=m<m_0$, if

$$
|\mathcal{E}(E_m)-\mathcal{E}(E)| < q(\epsilon)m^{\frac{n-1}{n}},
$$
there exists $x_0 \in \mathbb{R}^n$ such that  
$$
\frac{|(E+x_0) \Delta E_m|}{|E_m|}+\Big(\frac{|K|}{m}\Big)^{\frac{n-1}{n}} |\mathcal{F}(E)-\mathcal{F}(E_m)| < \epsilon.
$$
\end{thm}

\begin{thm} \label{@'r}
Assume $f(x)=|x|$, $\lambda \in (0,n)$, $n \ge 2$, and $E(m)$ does not have a minimizer when $m>m_*$. There is a modulus $a(\lambda, m,\epsilon)$ when $m<m_*$ such that if 
$$
\limsup_{m \rightarrow m_*^-} a(\lambda, m,\epsilon)>0,
$$
the critical mass conjecture \eqref{ar9} is true.
Also,
$$
a(\lambda, m,\epsilon) \le \max\Big\{\Big(\frac{m}{|K|}\Big)^{\frac{n-1}{n}}, \frac{\mathcal{H}^{n-1}(\mathbb{S}^{n-1})}{2(n-\lambda)}\Big(\frac{m}{|B|}\Big)^{\frac{n-\lambda}{n}}m \Big\} \epsilon.
$$
\end{thm}

In the statement of the aforementioned theorem, the assumption is imposed with  $E(m)$ not admitting a minimizer when $m>m_*$. The next theorem almost achieves that assumption in a few cases also inclusive of the most important case $n=3, \lambda \in (0,1]$, $f|_{\mathbb{S}^{n-1}}=1$. 

\begin{thm} \label{@'rj}
Suppose $\lambda \in (0,1]$.   Then $E(m)$ does not have a minimizer when $m>m_{\lambda,n,f|_{\mathbb{S}^{n-1}}}$ with
$$
m_{1,n, f|_{\mathbb{S}^{n-1}}}=\frac{2\max_{\nu\in \mathbb{S}^{n-1}}\{f(\nu)\}\mathcal{H}^{n-1}(\mathbb{S}^{n-1})}{|B^{n-1}|},
$$
and assuming $\lambda \in (0,1)$, 
\begin{align*}
m_{\lambda,n, f_{\mathbb{S}^{n-1}}}&=|B|\Big(\frac{|B|}{|B^{n-1}|}\max_{\nu\in \mathbb{S}^{n-1}}\{f(\nu)\}\frac{\mathcal{H}^{n-1}(\mathbb{S}^{n-1})}{\mathcal{D}_{\lambda-1}(B)}\Big)^{\frac{n}{n-\lambda+1}}
\end{align*}

\end{thm}

\begin{rem} \label{zw7}
One interesting remark is that 
$$
\limsup_{m \rightarrow m_*^-}\max\Big\{\Big(\frac{m}{|K|}\Big)^{\frac{n-1}{n}}, \frac{\mathcal{H}^{n-1}(\mathbb{S}^{n-1})}{2(n-\lambda)}\Big(\frac{m}{|B|}\Big)^{\frac{n-\lambda}{n}}m \Big\} \epsilon>0.
$$
In particular, if a lower bound is analogous, then the imposed assumption on the $\limsup$ in the Theorem \ref{@'r} is true. The lower bound is valid in a related problem \cite{qk}. The technique also implies $q(\epsilon) \le \alpha \epsilon$ supposing $\epsilon, m$ small. Note that in the case $n=3, \lambda=1$, $f|_{\mathbb{S}^{n-1}}=1$, $m_*=5\frac{2^{1/3}-1}{1-2^{-2/3}} \approx 3.512$, not too far from $m_{1,3, 1}=8,$ first obtained in \cite{qkn}. Furthermore, if $\lambda \in (0,1)$, $m_{\lambda,n, 1}$ has a similar numerical value as the conjectured $m_*$. In addition, one may also obtain an inequality analogous to the one in Theorem \ref{@'rj} for $\lambda \in (1,2]$ utilizing \cite{MR4314139}. 

\end{rem}

\noindent Acknowledgement: I want to thank Rupert Frank for his pertinent and meticulous remarks which led to an enhanced content.

\section{The proofs}

\subsection{Proof of Theorem \ref{@'}}

\begin{proof}
Assume the theorem is false. Then for all $m_0>0$, for all moduli $q$ there exists $m<m_0$ such that for a fixed $\epsilon_0 >0$ there exists $\epsilon<\epsilon_0$ $\&$ there exist $E_{m, \epsilon_0}, E_{m, \epsilon_0}' \subset B_R$, $|E_{m, \epsilon_0}|=|E_{m, \epsilon_0}'|=m$ such that

$$
|\mathcal{E}(E_{m,\epsilon_0})-\mathcal{E}(E_{m,\epsilon_0}')| < q(\epsilon)m^{\frac{n-1}{n}},
$$ 
and

\begin{equation} \label{wl4}
\inf_{x_0 \in \mathbb{R}^n} \frac{|(E_{m,\epsilon_0}'+x_0) \Delta E_{m,\epsilon_0}|}{|E_{m,\epsilon_0}|} +\Big(\frac{|K|}{m}\Big)^{\frac{n-1}{n}} |\mathcal{F}(E_{m,\epsilon_0}')-\mathcal{F}(E_{m,\epsilon_0})| \ge \epsilon>0.
\end{equation}
Let $m_0=\frac{1}{k}$, $w_k \rightarrow 0^+$, $\hat{q}$ a modulus of continuity and define 

\begin{equation} \label{t3}
q_k=w_k \hat{q}(\epsilon),
\end{equation}
hence there exists $m_k<\frac{1}{k}$ such that for a fixed $\epsilon_0 >0$ there exists $\epsilon<\epsilon_0$ $\&$ there exist $E_{m_k, \epsilon_0}, E_{m_k, \epsilon_0}' \subset B_R$, $|E_{m_k, \epsilon_0}|=|E_{m_k, \epsilon_0}'|=m_k<\frac{1}{k}$ such that
$$
|\mathcal{E}(E_{m_k,\epsilon_0})-\mathcal{E}(E_{m_k,\epsilon_0}')| < q_k m_k^{\frac{n-1}{n}},
$$ 
and
\begin{equation} \label{wlje1}
\inf_{x_0 \in \mathbb{R}^n} \frac{|(E_{m_k,\epsilon_0}'+x_0) \Delta E_{m_k,\epsilon_0}|}{|E_{m_k,\epsilon_0}|} +(\frac{|K|}{m_k})^{\frac{n-1}{n}} |\mathcal{F}(E_{m_k,\epsilon_0}')-\mathcal{F}(E_{m_k,\epsilon_0})|\ge \epsilon>0.
\end{equation}
Set 
\begin{equation} \label{a_k}
a_k=q_k m_k^{\frac{n-1}{n}},
\end{equation} 

$E_{m_k}=E_{m_k,\epsilon_0}$, $E_{m_k}'=E_{m_k,\epsilon_0}'$. Also, define $\gamma_k=(\frac{|K|}{m_k})^{\frac{1}{n}}$ such that
$$|\gamma_k E_{m_k}|=|K|.$$  \\

\noindent Claim 1:
\begin{equation} \label{theconv}
\delta(\gamma_k E_{m_k}):=\frac{\mathcal{F}(\gamma_k E_{m_k})}{n|K|^{\frac{1}{n}}|\gamma_k E_{m_k}|^{\frac{n-1}{n}}} - 1 \rightarrow 0
\end{equation}
as $k \rightarrow \infty$.\\

\noindent Proof of Claim 1:\\
Observe thanks to the anisotropic isoperimetric inequality \cite{MR1130601} that 
\begin{align*}
\mathcal{F}((1/\gamma_k)K)+\mathcal{D}(E_{m_k}) \le \mathcal{E}(E_{m_k}) \le \mathcal{E}((1/\gamma_k)K).
\end{align*}
Hence

\begin{align*}
\delta(\gamma_k E_{m_k})(n|K|)&=\gamma_k^{n-1}|\mathcal{F}(E_{m_k})-\mathcal{F}((1/\gamma_k)K)| \le \gamma_k^{n-1}\Big(\mathcal{D}((1/\gamma_k)K)-\mathcal{D}(E_{m_k})\Big)\\
&\le \gamma_k^{n-1}\frac{\mathcal{H}^{n-1}(\mathbb{S}^{n-1})}{2(n-\lambda)}(\frac{|(1/\gamma_k)K|}{|B|})^{\frac{n-\lambda}{n}} | (1/\gamma_k)K) \Delta E_{m_k}|\\
&\le a_* m_k^{2-\frac{\lambda}{n}-\frac{n-1}{n}} \rightarrow 0,
\end{align*}
(with the inequality in \cite{MR3322379} used in estimating $\mathcal{D}((1/\gamma_k)K)-\mathcal{D}(E_{m_k})$).

By the triangle inequality,

\begin{align*}
|\mathcal{F}(E_{m_k}')&-\mathcal{F}(E_{m_k})|\\
&=|[\mathcal{E}(E_{m_k}')-\mathcal{E}(E_{m_k})]+[\mathcal{D}(E_{m_k}) -\mathcal{D}(E_{m_k}')]|\\
&\le |\mathcal{E}(E_{m_k}')-\mathcal{E}(E_{m_k})|+\frac{\mathcal{H}^{n-1}(\mathbb{S}^{n-1})}{2(n-\lambda)}(\frac{|E_{m_k}|}{|B|})^{\frac{n-\lambda}{n}} | E_{m_k} \Delta E_{m_k}'|\\
&<a_k+\frac{\mathcal{H}^{n-1}(\mathbb{S}^{n-1})}{2(n-\lambda)}(\frac{m_k}{|B|})^{\frac{n-\lambda}{n}} | E_{m_k} \Delta E_{m_k}'|.
\end{align*}

Multiplying both sides by $\gamma_k^{n-1}$, 

\begin{align} 
|\mathcal{F}(\gamma_kE_{m_k}')&-\mathcal{F}(\gamma_kE_{m_k})|\\\label{co}
&<|K|^{\frac{n-1}{n}}\frac{a_k}{m_k^{\frac{n-1}{n}}}+|K|^{\frac{n-1}{n}}\frac{\mathcal{H}^{n-1}(\mathbb{S}^{n-1})}{n-\lambda}(\frac{m_k}{|B|})^{\frac{n-\lambda}{n}} m_k^{1-\frac{n-1}{n}} 
\end{align}
and since thanks to \eqref{t3} and \eqref{a_k}, 
$$a_k=q_k m_k^{\frac{n-1}{n}}=w_k \hat{q}(\epsilon)m_k^{\frac{n-1}{n}},$$
one obtains
$$
\frac{a_k}{m_k^{\frac{n-1}{n}}}=w_k \hat{q}(\epsilon) \rightarrow 0
$$
and hence via \eqref{co},
\begin{equation} \label{t5z}
 |\mathcal{F}(\gamma_kE_{m_k}')-\mathcal{F}(\gamma_kE_{m_k})| \rightarrow 0
\end{equation}
as $k \rightarrow \infty$.
Therefore
\begin{align}
\delta(\gamma_k E_{m_k}') & \le |\delta(\gamma_k E_{m_k}')-\delta(\gamma_k E_{m_k})|+\delta(\gamma_k E_{m_k}) \\
&= \frac{1}{n|K|}|\mathcal{F}(\gamma_kE_{m_k}')-\mathcal{F}(\gamma_kE_{m_k})|+\delta(\gamma_k E_{m_k}) \label{wlkj2}\\
&\hskip .15in \rightarrow 0 
\end{align}
as $k \rightarrow \infty$ from \eqref{theconv}.
Next \\

\noindent Claim 2: There exist $x_k, x_k' \in \mathbb{R}^n$ such that

\begin{equation} \label{wl5}
\frac{|(\gamma_kE_{m_k}+x_k) \Delta K|}{|\gamma_k E_{m_k}|} \rightarrow 0,
\end{equation}
\&
\begin{equation} \label{wlej}
\frac{|(\gamma_kE_{m_k}'+x_k') \Delta K|}{|\gamma_k E_{m_k}'|} \rightarrow 0
\end{equation}
as $k \rightarrow \infty$.\\

\noindent Proof of the Claim 2:  One can avoid the explicit rate in the anisotropic isoperimetric inequality \cite{MR2672283}  with a compactness argument:\\
if \eqref{wl5} is not true, then there is an $a>0$ such that up to a subsequence

\begin{equation} \label{zn}
\inf_x \frac{|(\gamma_kE_{m_k}+x) \Delta K|}{|\gamma_k E_{m_k}|} \ge a>0.
\end{equation}
Let  $E_k:=\gamma_k E_{m_k}$ $\&$ observe that 
$$
\sup_k \mathcal{F}(E_k)<\infty
$$
via \eqref{theconv}.
Hence the compactness for sets of finite perimeter then implies that up to a subsequence, 
\begin{equation} \label{Kwul}
(E_k+z_k) \rightarrow E \text{ in } L_{loc}^1,
\end{equation}
 $$
 |E| \le \liminf_k |E_k|= |K|.
$$
$$
\mathcal{F}(E) \le \liminf_k \mathcal{F}(E_k+z_k)=\liminf_k \mathcal{F}(E_k)=\mathcal{F}(K) \le \mathcal{F}(E);
$$
mod translations, one may find a radius $s>0$ such that
\begin{equation} \label{wp}
|E_k \cap B_s| \ge \min\{a_0|E_k|,a_1\}=\min\{a_0|K|,a_1\}=:a_2
\end{equation}
where $a_l>0$, $l=0,1$ (via e.g. the nucleation lemma \cite{MR3322379}). The $a_2$ is independent of $k$.
Assuming $|E|=|K|$, the cases of equality in the anisotropic isoperimetric inequality \cite{MR1130601} then imply $E=K+z$ for $z \in \mathbb{R}^n$, a contradiction with \eqref{zn}. Therefore $|E|<|K|$. Thanks to \eqref{wp}, $|E|>0$. Hence by concentration compactness \cite{MR3322379, MR20189283} there exists $\xi \in (0,1)$ such that for $\sigma>0$ and $k$ large, there exist $A_k^\sigma, N_k^\sigma \subset E_k$,

$$
|E_k \setminus(A_k^\sigma \cup N_k^\sigma)| <\sigma
$$  

$$
||A_k^\sigma|-\xi|K||<\sigma
$$

$$
||N_k^\sigma|-(1-\xi)|K||<\sigma,
$$

$$
\text{dist}(A_k^\sigma, N_k^\sigma) \rightarrow \infty,
$$
when $k \rightarrow \infty$.
Next  one may obtain a lower bound (see e.g. \cite{MR3425373}) 
$$
\mathcal{F}(E_k) \ge \mathcal{F}(A_k^\sigma)+\mathcal{F}(N_k^\sigma)-t_k,
$$
where
$$
t_k \rightarrow 0,
$$
and apply the anisotropic isoperimetric inequality
\begin{align*}
\mathcal{F}(E_k) &\ge \mathcal{F}(A_k^\sigma)+\mathcal{F}(N_k^\sigma)-t_k \\
&\ge \mathcal{F}(\alpha_{\sigma,k} K)+\mathcal{F}(\eta_{\sigma,k} K)-t_k,
\end{align*}
$$
\alpha_{\sigma,k}=(\frac{|A_k^\sigma|}{|K|})^{\frac{1}{n}},
$$
 
$$
\eta_{\sigma,k}=(\frac{|N_k^\sigma|}{|K|})^{\frac{1}{n}}.
$$
In particular, let $\sigma \rightarrow 0$,
$$
\frac{\mathcal{F}(E_k)}{\mathcal{F}(K)}\ge \Big(\frac{\mathcal{F}((\xi)^{\frac{1}{n}} K)}{\mathcal{F}(K)}+\frac{\mathcal{F}((1-\xi)^{\frac{1}{n}} K)}{\mathcal{F}(K)}-\frac{t_k}{\mathcal{F}(K)}\Big).
$$
Now taking $k \rightarrow \infty$,
\begin{equation}\label{dj}
1 \ge (1-\xi)^{\frac{n-1}{n}}+\xi^{\frac{n-1}{n}};
\end{equation}
set 
$$
h(s)=(1-s)^{\frac{n-1}{n}}+s^{\frac{n-1}{n}}
$$
when $s \in [0,1]$. Observe $h(0)=h(1)=1$, $h'(s)=0$ only if $s=1/2$, $h(1/2)=2^{\frac{1}{n}}>1$, therefore via $\xi \in (0,1)$,
$$
(1-\xi)^{\frac{n-1}{n}}+\xi^{\frac{n-1}{n}}=h(\xi)>1,
$$
contradicting \eqref{dj}. The claim \eqref{wl5} is therefore true. Similarly, from \eqref{wlkj2}, a symmetric argument implies
$$
\frac{|(\gamma_kE_{m_k}'+x_k') \Delta K|}{|\gamma_k E_{m_k}'|} \rightarrow 0
$$
as $k \rightarrow \infty$, therefore \eqref{wlej} is true. Hence the claim is valid.\\

\noindent Thus \eqref{wl5} and \eqref{wlej} yield $k \in \mathbb{N}$ such that, via the triangle inequality in $L^1$ applied to characteristic functions,
\begin{align*}
\frac{|(E_{m_k}'+\frac{(x_k'-x_k)}{\gamma_k}) \Delta E_{m_k}|}{|E_{m_k}|}&=\frac{|(\gamma_kE_{m_k}'+x_k')  \Delta (\gamma_kE_{m_k}+x_k) |}{|\gamma_k E_{m_k}|}\\
&\le \frac{|(\gamma_kE_{m_k}+x_k) \Delta K|}{|\gamma_k E_{m_k}|}+\frac{|K\Delta (\gamma_kE_{m_k}'+x_k')|}{|\gamma_k E_{m_k}|}\\
&<\frac{\epsilon}{2},
\end{align*}
a contradiction to

$$
(\frac{|K|}{m_k})^{\frac{n-1}{n}} |\mathcal{F}(E_{m_k})-\mathcal{F}(E_{m_k}')|= |\mathcal{F}(\gamma_kE_{m_k}')-\mathcal{F}(\gamma_kE_{m_k})| \rightarrow 0,
$$
see \eqref{t5z},
\begin{align*}
\frac{|(E_{m_k}'+\frac{(x_k'-x_k)}{\gamma_k}) \Delta E_{m_k}|}{|E_{m_k}|}\ge \frac{\epsilon}{2}>0,
\end{align*}
cf. \eqref{wlje1}. 
\end{proof}

\subsection{Proof of Theorem \ref{@'r}}

\begin{proof}
Theorem \ref{@'} initially implies that there exists $m_0>0$ such that $E_m$ is, up to a translation, unique for $m<m_0$. 
Next define
$$
\mathcal{M}=\sup \{m_0: \text{ $E_m$ is unique for all $m<m_0$}\}\le m_*<\infty.
$$
Thanks to Theorem 2 in \cite{MR4314139}, $\mathcal{M}=m_*$. \\

\noindent Claim: For $\epsilon>0$, $m<m_*$, there exists $w_m(\epsilon)>0$ such that if 
$|E|=|E_m|$, $E \subset B_R$, and 
$$
|\mathcal{E}(E)-\mathcal{E}(E_m)|<w_m(\epsilon),
$$
then there exists $x \in \mathbb{R}^n$ such that 
$$
\frac{|(E_m+x) \Delta E|}{|E_m|}+\Big(\frac{|K|}{m}\Big)^{\frac{n-1}{n}} |\mathcal{F}(E)-\mathcal{F}(E_m)|<\epsilon.
$$

\noindent Proof of Claim:\\
Initially, for $m \in (0,m_0)$, note that this is via Theorem \ref{@'}: the constraint $\epsilon \in (0,\epsilon_0(m))$ is not crucial; there is a constant $a_m$ such that $\epsilon \in (0,\epsilon_0(m))$ may be replaced with all $\epsilon>0$ and $q(\epsilon)m^{\frac{n-1}{n}}$ with
$$
a_m q(\epsilon).
$$
Hence suppose $m \in [m_0, m_*)$ and assume the statement is not true, then there exists $\epsilon>0$ and for $w>0$, there exist $E_w'$ and $E_m$, $|E_w'|=|E_m|=m$, 

$$
|\mathcal{E}(E_w')-\mathcal{E}(E_m)|<w
$$

$$
\inf_{x} \frac{|(E_m+x)\Delta E_w'|}{|E_m|} +\Big(\frac{|K|}{m}\Big)^{\frac{n-1}{n}} |\mathcal{F}(E_w')-\mathcal{F}(E_m)|\ge \epsilon.
$$
In particular, set $w=\frac{1}{k}$, $k \in \mathbb{N}$; observe that there exist $E_{\frac{1}{k}}'$, $|E_{\frac{1}{k}}'|=m$,

 $$
|\mathcal{E}(E_m)-\mathcal{E}(E_{\frac{1}{k}}')|<\frac{1}{k},
$$

$$
\inf_{x} \frac{|(E_m+x)\Delta E_{\frac{1}{k}}'|}{|E_m|} +\Big(\frac{|K|}{m}\Big)^{\frac{n-1}{n}} |\mathcal{F}(E_{\frac{1}{k}}')-\mathcal{F}(E_m)|\ge \epsilon.
$$
Thus

$$
\mathcal{F}(E_{\frac{1}{k}}') \le \mathcal{E}(E_{\frac{1}{k}}') < \frac{1}{k}+\mathcal{E}(E_m), 
$$

$E_{\frac{1}{k}}' \subset B_R$, and the compactness for sets of finite perimeter, up to a subsequence, thus implies 

$$
E_{\frac{1}{k}}' \rightarrow E' \hskip .3in in \hskip .08in L^1(B_R), 
$$
and therefore $|E'|=m$,

$$
\mathcal{E}(E') \le \liminf_k \mathcal{E}(E_{\frac{1}{k}}') =\mathcal{E}(E_m),
$$
which implies
$E'$ is a minimizer. Observe that since  $E_{\frac{1}{k}}' \rightarrow E'$ \& 
$$\mathcal{E}(E_{\frac{1}{k}}') \rightarrow \mathcal{E}(E'),$$ 
it is also true that 

$$
\mathcal{F}(E_{\frac{1}{k}}') \rightarrow \mathcal{F}(E') 
$$
contradicting
$$
\inf_{x} \frac{|(E_m+x)\Delta E'|}{|E_m|} +\Big(\frac{|K|}{m}\Big)^{\frac{n-1}{n}} |\mathcal{F}(E')-\mathcal{F}(E_m)|\ge \epsilon,
$$
via Theorem 2 in \cite{MR4314139} because the minimizer is a ball. In particular, the claim is true.\\
Set
$$
a(\lambda, m,\epsilon)=w_m(\epsilon).
$$
Suppose
$$
\limsup_{m \rightarrow m_*^-} a(\lambda, m,\epsilon)>0,
$$
and let $\{m_k\}$ be a sequence so 

$$
\limsup_{m \rightarrow m_*^-} a(\lambda, m,\epsilon)=\lim_{k \rightarrow \infty} a(\lambda, m_k,\epsilon).
$$
Via Theorem 2 in \cite{MR4314139}, the mass $m_*$ balls are minimizers. 
If there is a minimizer $E_{m_*}$ which is not a ball, when $B^{m_k}$ are balls with mass $m_k$ centered at the origin, define
$$
\epsilon=\inf_x \frac{|(x+B^{m_*}) \Delta E_{m_*}|}{|E_{m_*}|}>0.
$$
Set $|e_kE_{m_*}|=m_k$ so that 
$$
e_kE_{m_*} \rightarrow E_{m_*}.
$$
Observe in addition
$$
B^{m_k} \rightarrow B^{m_*}.
$$
Thanks to $e_k \rightarrow 1$,
$$
\mathcal{E}(e_kE_{m_*}) \rightarrow \mathcal{E}(E_{m_*});
$$
moreover
$$
\mathcal{E}(B^{m_k}) \rightarrow \mathcal{E}(B^{m_*}).
$$
Hence
$$
\lim_{k \rightarrow \infty} a(\lambda, m_k,\frac{\epsilon}{2})>0
$$
yields for $k$ sufficiently large,
\begin{align*}
|\mathcal{E}(e_kE_{m_*})-\mathcal{E}(B^{m_k})|&\le |\mathcal{E}(e_kE_{m_*})-\mathcal{E}(E_{m_*})|+|\mathcal{E}(B^{m_*})-\mathcal{E}(B^{m_k})|\\
&<w_{m_k}(\frac{\epsilon}{2}).
\end{align*}
Therefore, because minimizers are essentially bounded \cite{MR3055587, MR3272365}, $e_kE_{m_*} \subset B_R$ for an $R>0$ (uniformly in $k$), thus the claim yields $x_k \in \mathbb{R}^n$ such that 
$$
\frac{|(e_kE_{m_*}+x_k) \Delta B^{m_k}|}{|B^{m_k}|}<\frac{\epsilon}{2}.
$$
Next note that if $k$ is large,
$$
\frac{\epsilon}{2}>\frac{|(e_kE_{m_*}+x_k) \Delta B^{m_k}|}{|B^{m_k}|} \approx \frac{|(E_{m_*}+x_k) \Delta B^{m_*}|}{|B^{m_*}|} \ge \epsilon,
$$
a contradiction. Therefore this yields the $m=m_*$ endpoint uniqueness.\\

\noindent Now note that the modulus in the claim yields
$$
w_{m}(\frac{|(E+x_0) \Delta E_m|}{|E_m|}+\Big(\frac{|K|}{m}\Big)^{\frac{n-1}{n}} |\mathcal{F}(E)-\mathcal{F}(E_m)| ) \le |\mathcal{E}(E_m)-\mathcal{E}(E)| 
$$
because 
the continuity implies the existence of a non-decreasing $\phi: \mathbb{R}_+ \rightarrow \mathbb{R}_+$, so that
$$
\phi(\frac{|(E+x_0) \Delta E_m|}{|E_m|}+\Big(\frac{|K|}{m}\Big)^{\frac{n-1}{n}} |\mathcal{F}(E)-\mathcal{F}(E_m)| ) \le |\mathcal{E}(E_m)-\mathcal{E}(E)|. 
$$
Set 
$$\overline{w}_m(\epsilon)=\phi(\epsilon),$$
then if 
$$
|\mathcal{E}(E_m)-\mathcal{E}(E)|<\overline{w}_m(\epsilon),
$$
observe 
$$
\phi(\frac{|(E+x_0) \Delta E_m|}{|E_m|}+\Big(\frac{|K|}{m}\Big)^{\frac{n-1}{n}} |\mathcal{F}(E)-\mathcal{F}(E_m)| ) \le |\mathcal{E}(E_m)-\mathcal{E}(E)|<\overline{w}_m(\epsilon),
$$
and this implies
$$
\overline{w}_m(\frac{|(E+x_0) \Delta E_m|}{|E_m|}+\Big(\frac{|K|}{m}\Big)^{\frac{n-1}{n}} |\mathcal{F}(E)-\mathcal{F}(E_m)| )<\overline{w}_m(\epsilon).
$$
In particular,
$$
\frac{|(E+x_0) \Delta E_m|}{|E_m|}+\Big(\frac{|K|}{m}\Big)^{\frac{n-1}{n}} |\mathcal{F}(E)-\mathcal{F}(E_m)|<\epsilon.
$$
Hence one can select $w_m(\epsilon)=\phi(\epsilon).$ 
Thus
\begin{align*}
&w_{m}(\frac{|(E+x_0) \Delta E_m|}{|E_m|}+\Big(\frac{|K|}{m}\Big)^{\frac{n-1}{n}} |\mathcal{F}(E)-\mathcal{F}(E_m)| ) \\
&\le |\mathcal{E}(E_m)-\mathcal{E}(E)| \\
&\le |\mathcal{F}(E)-\mathcal{F}(E_m)|+|\mathcal{D}(E+x_0)-\mathcal{D}(E_m)|\\
&\le |\mathcal{F}(E)-\mathcal{F}(E_m)|+ \frac{\mathcal{H}^{n-1}(\mathbb{S}^{n-1})}{2(n-\lambda)}\Big(\frac{m}{|B|}\Big)^{\frac{n-\lambda}{n}}m \frac{| (E+x_0) \Delta E_{m}|}{m}\\
&\le\max\Big\{(\Big(\frac{|K|}{m}\Big)^{\frac{n-1}{n}})^{-1}, \frac{\mathcal{H}^{n-1}(\mathbb{S}^{n-1})}{2(n-\lambda)}\Big(\frac{m}{|B|}\Big)^{\frac{n-\lambda}{n}}m \Big\} \Big( \Big(\frac{|K|}{m}\Big)^{\frac{n-1}{n}}|\mathcal{F}(E)-\mathcal{F}(E_m)|+  \frac{| (E+x_0) \Delta E_{m}|}{m}\Big).\\
\end{align*}
Therefore supposing $s>0$, choose a set $E_s$ such that
$$
s=\frac{|(E_s+x_0) \Delta E_m|}{|E_m|}+\Big(\frac{|K|}{m}\Big)^{\frac{n-1}{n}} |\mathcal{F}(E_s)-\mathcal{F}(E_m)|. 
$$
Hence
$$
w_{m}(s) \le \max\Big\{(\Big(\frac{|K|}{m}\Big)^{\frac{n-1}{n}})^{-1}, \frac{\mathcal{H}^{n-1}(\mathbb{S}^{n-1})}{2(n-\lambda)}\Big(\frac{m}{|B|}\Big)^{\frac{n-\lambda}{n}}m \Big\} s,
$$

$$
a(\lambda, m,\epsilon) \le \max\Big\{\Big(\frac{m}{|K|}\Big)^{\frac{n-1}{n}}, \frac{\mathcal{H}^{n-1}(\mathbb{S}^{n-1})}{2(n-\lambda)}\Big(\frac{m}{|B|}\Big)^{\frac{n-\lambda}{n}}m \Big\} \epsilon.
$$

\end{proof}

\subsection{Proof of Theorem \ref{@'rj}}
In the proof of Theorem \ref{@'rj}, the technique in \cite{MR4314139} is applied.
\begin{lem}
Assume $n \ge 2$, $\lambda \in (0,n)$. Also, let $E$ be a minimizer for $E(m)$ with $m>0$. Then 
$$
\int \int_{E \times E} \frac{1}{|z-y|^{\lambda-1}}dz dy \le a_{n,f}|E|, 
$$
$$a_{n,f}=\frac{2\max_{\nu \in \mathbb{S}^{n-1}}\{f(\nu)\}\mathcal{H}^{n-1}(\mathbb{S}^{n-1})}{|B^{n-1}|}.$$
\end{lem}

\begin{proof}
If $\nu \in \mathbb{S}^{n-1}$, $t \in \mathbb{R}$, set 
$$
E_{\nu,t}^{\pm}=E \cap \{x: \pm \nu \cdot x > \pm t\}.
$$
Since when $\rho \ge 0$,
$$
|E_{\nu,t}^{+} \cup (E_{\nu,t}^{-}-\rho \nu)|=m,
$$
thanks to minimality, 
$$
\mathcal{E}(E_{\nu,t}^{+} \cup (E_{\nu,t}^{-}-\rho \nu)) \ge \mathcal{E}(E).
$$
Assuming $\rho>0$, 
\begin{align*}
\mathcal{F}(E_{\nu,t}^{+} \cup (E_{\nu,t}^{-}-\rho \nu))&=\mathcal{F}(E_{\nu,t}^{+})+\mathcal{F}(E_{\nu,t}^{-}-\rho \nu)\\
&= \int_{\partial^* E_{\nu,t}^{+} \cap \partial^* E} f(\nu_{E_{\nu,t}^{+}}) d\mathcal{H}^{n-1}+\int_{\{\nu \cdot x = t\} \cap E} f(-\nu) d\mathcal{H}^{n-1}\\
&+\int_{\partial^* E_{\nu,t}^{-} \cap \partial^* E} f(\nu_{E_{\nu,t}^{-}}) d\mathcal{H}^{n-1}+\int_{\{\nu \cdot x = t\} \cap E} f(\nu) d\mathcal{H}^{n-1} \\
&\le\mathcal{F}(E) + 2\max_{\nu \in \mathbb{S}^{n-1}}\{f(\nu)\}\mathcal{H}^{n-1}(\{\nu \cdot x = t\} \cap E),
\end{align*}
a.e. $t$. With $\rho \ge 0$, 
\begin{align*}
&\int \int_{\Big(E_{\nu,t}^{+} \cup (E_{\nu,t}^{-}-\rho \nu)\Big) \times \Big(E_{\nu,t}^{+} \cup (E_{\nu,t}^{-}-\rho \nu)\Big)}\frac{1}{|z-y|^{\lambda}}dz dy=\int \int_{E_{\nu,t}^{+} \times E_{\nu,t}^{+}}\frac{1}{|z-y|^{\lambda}}dz dy\\
&+\int \int_{E_{\nu,t}^{-} \times E_{\nu,t}^{-}}\frac{1}{|z-y|^{\lambda}}dz dy+2\int\int_{E_{\nu,t}^{+}\times E_{\nu,t}^{-}}\frac{1}{|z-y+\rho \nu|^{\lambda}}dz dy.
\end{align*}
Now taking $\rho \rightarrow \infty$,

\begin{align*}
\mathcal{F}(E)&+\frac{1}{2}\Big(\int \int_{E_{\nu,t}^{+} \times E_{\nu,t}^{+}}\frac{1}{|z-y|^{\lambda}}dz dy+\int \int_{E_{\nu,t}^{-} \times E_{\nu,t}^{-}}\frac{1}{|z-y|^{\lambda}}dz dy +2\int\int_{E_{\nu,t}^{+}\times E_{\nu,t}^{-}}\frac{1}{|z-y|^{\lambda}}dz dy\Big)\\
&\le \mathcal{F}(E)+2\max_{\nu \in \mathbb{S}^{n-1}}\{f(\nu)\}\mathcal{H}^{n-1}(\{\nu \cdot x = t\} \cap E)+\\
&\frac{1}{2}\Big(\int \int_{E_{\nu,t}^{+} \times E_{\nu,t}^{+}}\frac{1}{|z-y|^{\lambda}}dz dy+\int \int_{E_{\nu,t}^{-} \times E_{\nu,t}^{-}}\frac{1}{|z-y|^{\lambda}}dz dy \Big);\\
\end{align*}
therefore
\begin{align*}
\int\int_{E \times E} \chi_{\{\nu \cdot z>t>\nu \cdot y \}} \frac{1}{|z-y|^\lambda}dzdy&=\int\int_{E_{\nu,t}^{+}\times E_{\nu,t}^{-}}\frac{1}{|z-y|^{\lambda}}dz dy\\
& \le 2\max_{\nu \in \mathbb{S}^{n-1}}\{f(\nu)\}\mathcal{H}^{n-1}(\{\nu \cdot x = t\} \cap E).
\end{align*}
Integrating,
\begin{align*}
\int_{\mathbb{R}}\int\int_{E \times E}  \chi_{\{\nu \cdot z>t>\nu \cdot y \}}\frac{1}{|z-y|^\lambda}dzdydt &=\int\int_{E \times E} (\nu \cdot (z-y))_+\frac{1}{|z-y|^\lambda}dzdy\\
&\le \int_{\mathbb{R}} 2\max_{\nu \in \mathbb{S}^{n-1}}\{f(\nu)\}\mathcal{H}^{n-1}(\{\nu \cdot x = t\}\cap E )dt\\
&=2\max_{\nu \in \mathbb{S}^{n-1}}\{f(\nu)\}|E|.
\end{align*}
Observe that if $z-y \neq 0$, a rotation yields $\frac{z-y}{|z-y|}=e_n$, therefore 
define
$$
\Sigma(x_1,x_2,\ldots,x_{n-1})=(x_1,x_2,\ldots,x_{n-1},\sqrt{1-\sum_{l=1}^{n-1} x_l^2}),
$$
and observe that $\nu \cdot e_n=(\Sigma)_n$, $d\nu= \frac{1}{\sqrt{1-\sum_{l=1}^{n-1} x_l^2}}dx'$, $x=(x',x_n)$. Therefore 

$$
\int_{\mathbb{S}^{n-1}} (\nu \cdot \frac{z-y}{|z-y|})_+ d\nu=\int_{\mathbb{S}^{n-1}_+} (\nu \cdot e_n) d\nu=\int_{B^{n-1}} \sqrt{1-\sum_{l=1}^{n-1} x_l^2} \frac{1}{\sqrt{1-\sum_{l=1}^{n-1} x_l^2}}dx'=|B^{n-1}|.
$$
In particular
\begin{align*}
\int\int_{E \times E}|B^{n-1}|\frac{1}{|z-y|^{\lambda-1}}dzdy&=\int_{\mathbb{S}^{n-1}} \int\int_{E \times E} (\nu \cdot (z-y))_+\frac{1}{|z-y|^\lambda}dzdy d\nu \\
&\le 2\max_{\nu \in \mathbb{S}^{n-1}}\{f(\nu)\}|E|\mathcal{H}^{n-1}(\mathbb{S}^{n-1}).
\end{align*}
Hence let
$$
a_{n,f}=\frac{2\max_{\nu \in \mathbb{S}^{n-1}}\{f(\nu)\}\mathcal{H}^{n-1}(\mathbb{S}^{n-1})}{|B^{n-1}|}.
$$
\end{proof}
\begin{proof}
If $\lambda=1$,
$$
|E|^2 \le a_{n,f}|E|,
$$
in particular, $|E| \le a_{n,f}$.  Assume $\lambda \in (0,1)$.
Observe that for $t>0$, thanks to Riesz,
$$
\int \int_{E^* \times E^*} \frac{(1-e^{-t|z-y|^{1-\lambda}})}{t} dz dy \le \int \int_{E \times E} \frac{(1-e^{-t|z-y|^{1-\lambda}})}{t} dz dy
$$
where $E^*$ is the ball with $|E^*|=|E|$ and center at $0$.
Hence, letting $t \rightarrow 0^+$,
$$
\int \int_{E^* \times E^*} |z-y|^{1-\lambda} dz dy \le \int \int_{E \times E} |z-y|^{1-\lambda} dz dy \le \frac{2}{|B^{n-1}|}\max_{\nu \in \mathbb{S}^{n-1}}\{f(\nu)\}|E|\mathcal{H}^{n-1}(\mathbb{S}^{n-1})
$$
$$
\int \int_{E^* \times E^*} |z-y|^{1-\lambda} dz dy=2(\frac{|E|}{|B|})^{\frac{2n-\lambda+1}{n}} \mathcal{D}_{\lambda-1}(B)
$$
$$
|E| \le \Big((\frac{|B|^{\frac{2n-\lambda+1}{n}}}{\mathcal{D}_{\lambda-1}(B)}) \Big(\frac{1}{|B^{n-1}|}\max_{\nu \in \mathbb{S}^{n-1}}\{f(\nu)\}\mathcal{H}^{n-1}(\mathbb{S}^{n-1})\Big)\Big)^{\frac{n}{n-\lambda+1}}.
$$
\end{proof}

\newpage
\bibliographystyle{amsalpha}
\bibliography{References}

\end{document}